\documentclass[a4paper,12pt, twoside, reqno]{amsart}
\usepackage{amssymb}
\usepackage{amsmath}
\newtheorem{theorem}{Theorem}
\newtheorem{corollary}[theorem]{Corollary}

\newtheorem{defn}{Definition}
\newtheorem{ex}{Example}
\newtheorem{remark}{Remark}[section]

\title{Approximating fixed points of enriched contractions in Banach spaces}
\author{Vasile BERINDE$^{1,2}$}
\author{M\u ad\u alina P\u ACURAR$^{3}$}
\begin{document}
\maketitle \pagestyle{myheadings} \markboth{Vasile Berinde and M\u ad\u alina P\u ACURAR} {Enriched Kannan type...}
\begin{abstract}
We introduce a large class of mappings, called enriched contractions, which includes, amongst many other contractive type mappings, the Picard-Banach contractions and some nonexpansive mappings. We show that any enriched contraction has a unique fixed point and that this fixed point can be approximated by means of an appropriate Kransnoselskij iterative scheme. Several important results in fixed point theory are shown to be corollaries or consequences of the main results in this paper. We also study the fixed points of local enriched contractions, asymptotic enriched contractions and Maia type enriched contractions. Examples to illustrate the generality of our new concepts and the corresponding fixed point theorems are also given.

\end{abstract}

\section{Introduction}
Let $X$ be a nonempty set and $T:X\rightarrow X$ a self mapping. We denote the set of fixed points of $T$ by $Fix\,(T)$, i.e., $Fix\,(T)=\{a\in X: T(a)=a\}$ and define the $n^{th}$ iterate of $T$ as usually, that is, $T^0=I$ (identity map) and $T^n=T^{n-1}\circ T$, for $n\geq 1$.

$T$ is said to be a {\it Picard operator}, see for example Rus \cite{Rus01}, if\\ (i) $Fix\,(T)=\{p\}$; (ii)  $T^n(x_0) \rightarrow p$ as $n\rightarrow \infty$, for any $x_0$ in $X$.

The most useful class of Picard operators, that play a crucial role in nonlinear analysis, is the  class of mappings known in literature as {\it Picard-Banach contractions}, first introduced by Banach in \cite{Ban22}, in the case of what we call now a Banach space, and then extended to complete metric spaces by Caccioppoli  \cite{Cac}. 

Banach contraction mapping principle essentially states that, in a complete metric space $(X,d)$, any contraction $T:X\rightarrow X$, that is, any mapping for which there exists $c\in[0,1)$ such that
\begin{equation} \label{eq1}
d(Tx,Ty)\leq c\cdot d(x,y),\forall x,y \in X,
\end{equation}
is a Picard operator.

Starting from Picard-Banach fixed point theorem, a very impressive literature has developed in the last nine decades or so, see for example the monographs \cite{Ber07}, \cite{Rus08} and references therein. Moreover, the Picard-Banach fixed point theorem itself and some of its extensions have become very useful and versatile tools in solving various nonlinear problems: differential equations, integral equations, integro-differential equations, optimisation problems, variational inequalities etc., see \cite{Ber07}, \cite{Rus08}, \cite{Zei85} and \cite{Zei86} and the extensive literature cited there. 

Having in view both the theoretical and application importance of the Banach contraction mapping principle, the aim of this paper in to introduce a larger class of Picard operators, called enriched Banach contractions, that includes Picard-Banach contractions as a particular case and also includes some nonexpansive mappings. 

We then study the existence and uniqueness of fixed points of enriched Banach contractions and prove a strong convergence theorem for Kransnoselskij iteration used to approximate the fixed points of enriched Banach contractions. We also consider a local variant of Picard-Banach fixed point theorem and extend further the main results to the case of asymptotic enriched Banach contractions. Examples to illustrate the generality of our new results are also presented.

\section{Approximating fixed points of enriched contractions} 

\begin{defn}
Let $(X,\|\cdot\|)$ be a linear normed space. A mapping $T:X\rightarrow X$ is said to be an {\it enriched contraction} if there exist $b\in[0,ü\infty)$ and $\theta\in[0,b+1)$ such that
\begin{equation} \label{eq3}
\|b(x-y)+Tx-Ty\|\leq \theta \|x-y\|,\forall x,y \in X.
\end{equation}
To indicate the constants involved in \eqref{eq3} we shall also call  $T$ a  $(b,\theta$)-{\it enriched contraction}. 
\end{defn}

\begin{ex}[ ] \label{ex1}
\indent

(1) Any contraction $T$ with contraction constant $c$ is a $(0,c)$-enriched contraction, i.e., $T$ satisfies \eqref{eq3} with $b=0$ and $\theta=c$.

(2) Let $X=[0,1]$ be endowed with the usual norm and let $T:X\rightarrow X$ be defined by $Tx=1-x$, for all $x\in [0,1]$. Then $T$ is nonexpansive (it is an isometry),  $T$ is not a contraction but $T$ is an enriched contraction. Indeed, if $T$ would be a contraction then, by \eqref{eq1}, there would exist $c\in[0,1)$ such that 
$$
|x-y|\leq c \cdot |x-y|,\forall x,y \in [0,1],
$$
which, for any $x\neq y$, leads to the contradiction $1\leq c < 1$. 

On the other hand, the enriched contraction condition \eqref{eq3} is in this case equivalent to
$$
|(b-1)(x-y)|\leq  \theta \cdot |x-y|,\forall x,y \in [0,1],
$$
with $\theta\in[0,b+1)$.  For $x\neq y$ the previous inequality holds true if one chooses $0<b<1$ and $\theta=1-b$.  

Hence, for any $b\in(0,1)$, $T$ is a $(b,1-b)$-enriched contraction. Note also that $Fix\,(T)=\left\{\dfrac{1}{2}\right\}$.

\end{ex}
\begin{remark}
We note that for $T$ in Example \ref{ex1} (2), Picard iteration $\{x_n\}$ associated to $T$, that is, the sequence $x_{n+1}=1-x_n$, $n\geq 0$, does not converge for any initial guess $x_0$ different of $\dfrac{1}{2}$, the unique fixed point to $T$. 

This suggests us that, in order to approximate fixed points of enriched contractions we need more elaborate fixed point iterative schemes, like  Krasnoselskij iterative method, for which we prove in the following a strong convergence theorem in the class of {\it enriched} contractions.
\end{remark}

\begin{theorem}  \label{th1}
Let $(X,\|\cdot\|)$ be a Banach space and $T:X\rightarrow X$ a $(b,\theta$)-{\it enriched contraction}. Then

$(i)$ $Fix\,(T)=\{p\}$;

$(ii)$ There exists $\lambda\in (0,1]$ such that the iterative method
$\{x_n\}^\infty_{n=0}$, given by
\begin{equation} \label{eq3a}
x_{n+1}=(1-\lambda)x_n+\lambda T x_n,\,n\geq 0,
\end{equation}
converges to p, for any $x_0\in X$;

$(iii)$ The following estimate holds
\begin{equation}  \label{3.2-1}
\|x_{n+i-1}-p\| \leq\frac{c^i}{1-c}\cdot \|x_n-
x_{n-1}\|\,,\quad n=0,1,2,\dots;\,i=1,2,\dots,
\end{equation}
where $c=\dfrac{\theta}{b+1}$.
\end{theorem}

\begin{proof}
We split the proof into two different cases.

{\bf Case 1.} $b>0$. In this case, let us denote $\lambda=\dfrac{1}{b+1}$. Obviously, $0<\lambda<1$ and the enriched contractive condition \eqref{eq3} becomes
$$
\left \|\left(\frac{1}{\lambda}-1\right)(x-y)+Tx-Ty\right\|\leq \theta \|x-y\|,\forall x,y \in X,
$$
which can be written in an equivalent form as
\begin{equation} \label{eq5}
\|T_\lambda x-T_\lambda y\|\leq c \cdot \|x-y\|,\forall x,y \in X,
\end{equation}
where we denoted $c=\lambda \theta$ and
\begin{equation} \label{eq4}
T_\lambda (x)=(1-\lambda)x+\lambda T(x), \forall  x \in X.
\end{equation}
Since $\theta\in(0,b+1)$, it follows that $c\in(0,1)$ and therefore inequality \eqref{eq4} shows that $T_\lambda$ is a $c$-contraction. 

Note, in passing, that  $T$ and $T_\lambda$ are linked by the following important property:
\begin{equation} \label{eq4a}
Fix(\,T_\lambda)=Fix\,(T).
\end{equation}
In view of \eqref{eq4}, the Krasnoselskij iterative process $\{x_n\}^\infty_{n=0}$ defined by  \eqref{eq3a} is exactly the Picard iteration associated to $T_\lambda$, that is,
\begin{equation} \label{eq4b}
x_{n+1}=T_\lambda x_n,\,n\geq 0.
\end{equation}
Take $x=x_n$ and $y=x_{n-1}$ in  \eqref{eq5} to get
\begin{equation} \label{eq6}
\|x_{n+1}-x_{n}\|\leq c \cdot \|x_{n}-x_{n-1}\|,\,n\geq 1.
\end{equation}
By \eqref{eq6} one obtains routinely the following two estimates
\begin{equation} \label{eq7a}
\|x_{n+m}-x_{n}\|\leq c^n \cdot \frac{1-c^m}{1-c}\cdot \|x_{1}-x_{0}\|,\,n\geq 0, m\geq 1.
\end{equation}
and
\begin{equation} \label{eq8}
\|x_{n+m}-x_{n}\|\leq c \cdot \frac{1-c^m}{1-c}\cdot \|x_{n}-x_{n-1}\|,\,n\geq 1, \,m\geq 1.
\end{equation}
Now, by \eqref{eq7a} it follows that $\{x_n\}^\infty_{n=0}$ is a Cauchy sequence and hence it is convergent in the Banach space $(X,\|\cdot\|)$. Let us denote
\begin{equation} \label{eq10a}
p=\lim_{n\rightarrow \infty} x_n.
\end{equation}
By letting  $n\rightarrow \infty$ in \eqref{eq4b} and using the continuity of $T_\lambda$ we immediately obtain
$$
p=T_\lambda p,
$$
that is, $p\in Fix\,(T_\lambda)$. 

Next, we prove that $p$ is the unique fixed point of $T_\lambda$. Assume that $q\neq p$ is another fixed point of $T_\lambda$. Then, by \eqref{eq5} 
$$
0<\|p-q\|\leq c \cdot \|p-q\|<\|p-q\|,
$$
a contradiction. Hence $Fix\,(T_\lambda)=\{p\}$ and since, by \eqref{eq4b},  $Fix\,(T)=Fix(\,T_\lambda)$, claim $(i)$ is proven.

Conclusion $(ii)$ follows by \eqref{eq10a}.

To prove $(iii)$, we let $m\rightarrow \infty$ in \eqref{eq7a}  and \eqref{eq8} to get
\begin{equation} \label{eq11}
\|x_n-p\|\leq  \frac{c^n}{1-c}\cdot \|x_{1}-x_{0}\|,\,n\geq 1
\end{equation}
and
\begin{equation} \label{eq12}
\|x_n-p\|\leq \frac{c}{1-c}\cdot \|x_{n}-x_{n-1}\|,\,n\geq 1,
\end{equation}
respectively, where $c=\dfrac{\theta}{b+1}$. Now one can merge \eqref{eq11}  and \eqref{eq12} to get the unifying error estimate \eqref{3.2-1}.

{\bf Case 2.} $b=0$. In this case,  $\lambda=1$, $c=\theta$ and we proceed like in Case 1 but with $T(=T_1)$ instead of $T_\lambda$, when Kasnoselskij iteration \eqref{eq3a} reduces in fact to the simple Picard iteration associated to $T$,
$$
x_{n+1}=T x_n,\,n\geq 0.
$$
\end{proof}

\begin{remark}
1) In the particular case $b=0$, by Theorem \ref{th1} we get the classical Banach contraction fixed point theorem (see Banach \cite{Ban22}) in the setting of a Banach space.

2) Enriched Kannan contractions have been introduced and studied by Berinde and P\u acurar in \cite{Ber19e}.  Although the mapping $T$ in Example \ref{ex1} (2) is simultaneously an enriched Kannan mapping and an enriched contraction, it is easily seen that the class of enriched Kannan contractions  is also independent of the class of enriched Banach contractions, due to the fact that, see for example \cite{Rho}, the class of Kannan contractions  is independent of the class of Banach contractions. 

We shall call a mapping $T$ to be a strictly enriched Kannan mapping if it is not a usual Kannan mapping. So, an open problem would be to find a strictly enriched Kannan mapping which is not an enriched contraction.
\end{remark}

\section{Local and asymptotic versions of enriched contraction mapping principle} 

The Picard-Banach contraction mapping principle has a useful local version, see for example \cite{Gran},  which involves an open ball $B$ in a complete metric space $(X,d)$ and a nonself contraction map of $B$ into $X$ which has the essential property that does not displace the centre of the ball too far. The analogue of this result in the case of  enriched contractions is given by the following Corollary.

\begin{corollary}\label{cor1}
Let $(X,\|\cdot\|)$ be a Banach space, $B=B(x_0,r):=\{x\in X:\|x-x_0\|<r\}$, $r>0$ and let $T:B\rightarrow X$ be a $(b,\theta$)-enriched contraction. If $\|Tx_0-x_0\|<(b+1-\theta)r$, then $T$ has a fixed point.
\end{corollary}

\begin{proof}
W can choose $\varepsilon<r$ such that 
$$
\|Tx_0-x_0\|\leq (b+1-\theta)\varepsilon<(b+1-\theta)r. 
$$
On the other hand, as $T$ is a $(b,\theta$)-enriched contraction, there exist $b\in[0,ü\infty)$ and $\theta\in[0,b+1)$ such that
\begin{equation} \label{eq3a}
\|b(x-y)+Tx-Ty\|\leq \theta \|x-y\|,\forall x,y \in B.
\end{equation}
If $b>0$, we denote $\lambda=\dfrac{1}{b+1}$ which implies $0<\lambda<1$. So, the enriched contractive condition \eqref{eq3a} becomes
$$
\left \|\left(\frac{1}{\lambda}-1\right)(x-y)+Tx-Ty\right\|\leq \theta \|x-y\|,\forall x,y \in B,
$$
which can be written in an equivalent form as
\begin{equation} \label{eq5a}
\|T_\lambda x-T_\lambda y\|\leq c \cdot \|x-y\|,\forall x,y \in B,
\end{equation}
where we denoted $c=\lambda \theta=\dfrac{\theta}{b+1}$ and
\begin{equation} \label{eq4x}
T_\lambda (x)=(1-\lambda)x+\lambda T(x), \forall  x \in B.
\end{equation}
Since $\theta\in(0,b+1)$, it follows that $c\in(0,1)$ and therefore inequality \eqref{eq4x} shows that $T_\lambda$ is a $c$-contraction on $B$. 

Note also that the inequality $
\|Tx_0-x_0\|\leq (b+1-\theta)\varepsilon  
$
can be written equivalently as $\|T_\lambda x_0-x_0\|\leq \left(1-\dfrac{\theta}{b+1}\right)\varepsilon=(1-c)\varepsilon$.

We now prove that the closed ball $\overline{B}:=\{x\in X:\|x-x_0\|\leq \varepsilon\}$ is invariant under $T_\lambda$. Indeed, for any $x\in \overline{B}$ we have
$$
\|T_\lambda(x)-x_0\|\leq \|T_\lambda(x)-T_\lambda(x_0)\|+\|T_\lambda(x_0)-x_0\|
$$
$$
\leq  c \|x-x_0\|+(1-c)\varepsilon \leq c \varepsilon+(1-c)\varepsilon=\varepsilon,
$$
which proves that $T_\lambda (x)\in \overline{B}$, for any $x \in \overline{B}$.

Since $\overline{B}$ is complete, conclusion follows by Theorem \ref{th1}.

If $b=0$, we proceed in a similar way but we shall use the contraction mapping principle to get the final conclusion.
\end{proof}

\begin{remark}
In the particular case $b=0$, by Corollary \ref{cor1} we obtain the local variant of the contraction mapping principle, see for example Corollary (1.2) in \cite{Gran}.
\end{remark}

The following example shows that  there exist mappings $T$ which are not contractions but a certain iterate of them is  a contraction.

\begin{ex} (Examples 1.3.1, \cite{Rus01a}) \label{ex2}
Let $X=\mathbb{R}$ and $T:X\rightarrow X$ be given by $Tx=0$, if $x\in (-\infty,2]$ and $Tx=-\dfrac{1}{3}$, if $x\in (2,+\infty)$. Then $T$ is not a contraction (being discontinuous) but $T^2$ is a contraction.
\end{ex}

In such a case, we cannot apply the classical Picard-Banach contraction mapping principle and thus the following fixed point theorem is useful, see for example Theorem 1.3.2 in \cite{Rus01a}.

\begin{theorem}\label{th2}
Let $(X,d)$ be a complete metric space and let $T:X\rightarrow X$ be a mapping. If there exists a positive integer $N$ such that $T^N$ is a contraction, then $Fix\,(T)=\{x^*\}$.
\end{theorem}

Our first aim in this section is to obtain a similar result for the more general case of enriched contractions in the setting of a Banach space.

\begin{theorem}\label{th3}
Let $(X,\|\cdot\|)$ be a Banach space and let $U:X\rightarrow X$ be a mapping with the property that there exists a positive integer $N$ such that $U^N$ is a $(b,\theta$)-enriched contraction. Then 

$(i)$ $Fix\,(U)=\{p\}$;

$(ii)$ There exists $\lambda\in (0,1]$ such that the iterative method
$\{x_n\}^\infty_{n=0}$, given by
$$
x_{n+1}=(1-\lambda)x_n+\lambda U^N x_n,\,n\geq 0,
$$
converges to $p$, for any $x_0\in X$.
\end{theorem}

\begin{proof}
We apply Theorem \ref{th1} (i) for the mapping $T=U^N$ and obtain that $Fix\,(U^N)=\{p\}$. We also have
$$
U^N(U(p))=U^{N+1}(p)=U(U^N(p))=U(p),
$$
which shows that $U(p)$ is a fixed point of $U^N$. But $U^N$ has a unique fixed point, $p$, hence $U(p)=p$ and so $p\in Fix\,(U)$. 

The remaining part of the proof follows by Theorem \ref{th1}. 
\end{proof}

One of the most interesting generalizations of contraction mapping principle is the so-called Maia fixed point theorem, see \cite{Maia}, which was obtained by splitting the assumptions in the contraction mapping principle among two metrics defined on the same set. Its statement reads as follows.

 \begin{theorem} (Theorem 1.3.10, \cite{Rus01a})\label{th4}
 Let $X$ be a nonempty set, $d$ and $\rho$ two metrics on $X$ and $T : X \rightarrow X$ a mapping. Suppose that
 
 (i) $d(x,y)\leq \rho(x,y) $, for each $x,y \in X$;
 
 (ii) $(X, d)$ is a complete metric space;
 
 (iii) $T : X \rightarrow X$ is continuous  with respect to the metric $d$;
 
 (iv) $T$ is a contraction mapping with respect to the metric $\rho$. 
 
 Then $T$ is a Picard operator.
 \end{theorem}

The next theorem is an extension of Theorem \ref{th4} to the class of enriched contractions but is established in the particular case of a Banach space.

\begin{theorem} (Theorem 1.3.10, \cite{Rus01a})\label{th5}
 Let $X$ be a linear space, $\|\cdot \|_d$ and $\|\cdot\|_{\rho}$ two norms on $X$ and $T : X \rightarrow X$ a mapping. Suppose that
 
 (i) $\|x-y\|_d\leq \|x-y\|_{\rho} $, for each $x,y \in X$;
 
 (ii) $(X, \|\cdot \|_d)$ is a Banach space;
 
 (iii) $T : X \rightarrow X$ is continuous  with respect to the norm $\|\cdot\|_{\rho}$;
 
 (iv) $T$ is a $(b,\theta)$-enriched contraction mapping with respect to the norm $\|\cdot\|_{\rho}$. 
 
 Then $T$ is a Picard operator.
 \end{theorem}

\begin{proof}
Let $x_0\in X$. By (iv), we deduce similarly to the proof of Theorem \ref{th1} that $\{T_{\lambda}^n(x_0)\}$ is a Cauchy sequence in $(X,\|\cdot\|_{\rho})$, where as usually $\lambda=\dfrac{1}{b+1}$. By (i), $\{T_{\lambda}^n(x_0)\}$ is a Cauchy sequence in $(X,d)$ and by (ii) it converges. Let
$$
x^*=\lim_{n\rightarrow \infty} T_{\lambda}^n(x_0).
$$
By (iii) we obtain that $x^*\in Fix\,(T_{\lambda})$ and by (iv) that $Fix\,(T_{\lambda})=\{x^*\}$. Since $Fix\,(T_{\lambda})=Fix\,(T)$, the conclusion follows.
\end{proof}

\section{Conclusions}

1. We introduced the class of enriched contractions, which includes the Picard-Banach contractions as a particular case and also certain nonexpansive mappings. 

2. We have shown that any enriched contraction has a unique fixed point that can be approximated by means of some Kransnoselskij iterations. 

3. We presented  relevant examples to show that the class of enriched contractions strictly includes the Picard-Banach contractions in the sense that there exists mappings which are not contractions and belong to the class of enriched contractions.

4. It is worth to mention that enriched contractions preserve a fundamental property of Picard-Banach  contractions: any enriched contraction has a unique fixed point and is continuous.

5. It is well known that the class of Picard-Banach contractions is independent of the class of Kannan contractions (which are in general discontinuous mappings) but, in view of Example \ref{ex1}, there exist enriched Banach contractions which are simultaneously enriched Kannan contractions. 

6. We also obtained a local version (Corollary \ref{cor1}) as well as an asymptotic version  (Theorem \ref{th2}) of the enriched contraction mapping principle  (Theorem \ref{th1}).

7. Last but not least, we obtained a fixed point theorem of Maia type which extends Maia fixed point theorem (Theorem \ref{th4}) from the class of Picard-Banach contractions to the class of enriched contractions  (Theorem \ref{th5}).

8. Picard-Banach contractions have been previously extended in \cite{Ber04} to the so called {\it almost contractions}, a class of mappings which preserve most of the features of  Picard-Banach contractions but may have more than one fixed point, see also \cite{Alg}, \cite{Ber12}, \cite{BerP} and references therein. Our aim is to unify the class of almost contractions and enriched contractions in a future work.

\vskip 0.5 cm {\it $^{1}$ Department of Mathematics and Computer Science

North University Center at Baia Mare

Technical University of Cluj-Napoca 

Victoriei 76, 430122 Baia Mare ROMANIA

E-mail: vberinde@cunbm.utcluj.ro}

\vskip 0.5 cm {\it $^{2}$ Academy of Romanian Scientists  (www.aosr.ro)

E-mail: vasile.berinde@gmail.com}

\vskip 0.5 cm {\it $^{3}$ Department of Statistics, Analysis, Forecast and Mathematics

Faculty of Economics and Bussiness Administration 

Babe\c s-Bolyai University of Cluj-Napoca,  Cluj-Napoca ROMANIA

E-mail: madalina.pacurar@econ.ubbcluj.ro}

\end{document}